\documentclass[12pt]{amsart}
\usepackage{verbatim}
\usepackage{amssymb}
\usepackage{upref}
\usepackage[cmtip,all]{xy}
\usepackage[usenames,dvipsnames]{color}
\usepackage[normalem]{ulem}

\allowdisplaybreaks

\newtheorem{thm}{Theorem}[section]
\newtheorem{lem}[thm]{Lemma}

\newtheorem{prop}[thm]{Proposition}

\theoremstyle{definition} % definition style
\newtheorem{defn}[thm]{Definition}
\newtheorem{q}[thm]{Question}

\newtheorem{rem}[thm]{Remark}

\numberwithin{equation}{section}

\newcommand{\secref}[1]{Section~\textup{\ref{#1}}}

\newcommand{\thmref}[1]{Theorem~\textup{\ref{#1}}}

\newcommand{\lemref}[1]{Lemma~\textup{\ref{#1}}}
\newcommand{\propref}[1]{Proposition~\textup{\ref{#1}}}

\newcommand{\remref}[1]{Remark~\textup{\ref{#1}}}

\newcommand{\DD}{\mathcal D}

\renewcommand{\AA}{\mathcal A}
\newcommand{\BB}{\mathcal B}

\newcommand{\II}{\mathcal I}
\newcommand{\RR}{\mathcal R}

\newcommand{\C}{\mathbb C}

\newcommand{\variso}{\overset{\simeq}{\longrightarrow}}

\newcommand{\what}{\widehat}

\newcommand{\wilde}{\widetilde}

\newcommand{\ann}{^\perp}

\renewcommand{\epsilon}{\varepsilon}

\renewcommand{\:}{\colon}
\renewcommand{\subset}{\subseteq}

\newcommand{\xm}{\otimes_{\max}}
\newcommand{\xg}{\otimes_G}
\newcommand{\elg}{\ell^\infty(G)}

\renewcommand{\)}{\textup)}
\newcommand{\rt}{\textup{rt}}
\newcommand{\lt}{\textup{lt}}
\newcommand{\id}{\text{\textup{id}}}

\newcommand{\midtext}[1]{\quad\text{#1}\quad}
\newcommand{\righttext}[1]{\quad\text{#1 }}

\begin{document}
\title[Tensor-product coaction functors]{Tensor-product coaction functors}
\author[Kaliszewski]{S. Kaliszewski}
\address{School of Mathematical and Statistical Sciences
\\Arizona State University
\\Tempe, Arizona 85287}
\email{kaliszewski@asu.edu}
\author[Landstad]{Magnus~B. Landstad}
\address{Department of Mathematical Sciences\\
Norwegian University of Science and Technology\\
NO-7491 Trondheim, Norway}
\email{magnus.landstad@ntnu.no}
\author[Quigg]{John Quigg}
\address{School of Mathematical and Statistical Sciences
\\Arizona State University
\\Tempe, Arizona 85287}
\email{quigg@asu.edu}

\date{December 10, 2019}
%\date{\today}


\dedicatory{Dedicated to the memory of J. M. G. Fell, 1923--2016.}

\begin{abstract} 
Recent work by Baum, Guentner, and Willett,
and further developed by Buss, Echterhoff, and Willett
introduced a crossed-product functor that involves tensoring an action with a fixed action $(C,\gamma)$,
then forming the image inside the crossed product of
the maximal-tensor-product action.
For discrete groups, we
give an analogue,
for coaction functors.
We prove that composing our tensor-product coaction functor
with the full crossed product of an action reproduces their
tensor-crossed-product functor.
We prove that every such tensor-product coaction functor is exact and if
$(C,\gamma)$ is the action by translation on $\ell^\infty(G)$,
we prove that
the associated tensor-product coaction functor is minimal;
thereby recovering the analogous result by the above authors.
Finally, we discuss the connection with the $E$-ization functor we defined earlier, 
where $E$ is a large ideal of $B(G)$.
\end{abstract}
\maketitle

\section{Introduction}\label{intro}

For a fixed locally compact group $G$,
the full and the reduced crossed-product functors each take an action of $G$ on a $C^*$-algebra and produce
a $C^*$-algebra.
Baum, Guentner, and Willett \cite{bgwexact} studied \emph{exotic crossed-product functors} that are intermediate between the full and reduced crossed products, as part of an investigation of the Baum-Connes conjecture.
In Section~5 of that paper, the authors introduced a natural class of crossed products arising from tensoring with a fixed action.
Their general construction starts with an arbitrary crossed-product functor, but we only need the version for the full crossed product.
They prove that their tensor-crossed-product functor is 
exact.
Buss, Echterhoff, and Willett \cite{bew} further the study of these tensor-crossed-product functors,
and in Section~9 of that paper they prove that
the case with $\ell^\infty(G)$ produces the smallest of all tensor-crossed-product functors.
This leads them to ask whether tensoring with $\ell^\infty(G)$ in fact produces the minimal exact
correspondence crossed product.

Thus the $\ell^\infty(G)$-tensor-crossed-product functor takes on substantial importance.
We have initiated in \cite{klqfunctor, klqfunctor2} a new approach to exotic crossed products,
applying a coaction functor to the full crossed product.
We have shown that this procedure reproduces many
(perhaps all of the important?)
crossed-product functors,
and we believe that fully utilizing the coactions makes for a more robust theory.
In \cite{klqfunctor, klqfunctor2} we
have shown that the theory of coaction functors
is in numerous aspects parallel to that of
the crossed-product functors of
\cite{bgwexact, bew}.
In this paper,
using the techniques of \cite{bgwexact,bew} as a guide,
we initiate an investigation into an analogue for coaction functors of the tensor-crossed-product functors for actions
(see \secref{tensor D} for details).
Our development must of course have many differences from that of crossed products by actions,
since coactions are different from actions,
and also,
according to our paradigm for crossed-product functors,
our coaction functors form the second part of such a crossed product.

To give a more precise overview of our tensor-coaction functors,
we first outline
(with slightly modified notation),
the construction of tensor-crossed-products
from
\cite{bgwexact,bew}.
For technical reasons, our techniques currently only apply to discrete groups,
so from now on we suppose that the group $G$ is discrete.%
\footnote{It is certainly a draw-back of our techniques in this paper that we only handle discrete groups.
It is imperative to find some way to extend all this to arbitrary locally compact 
$G$, and we will investigate this in future research.}
Fix an action $(C,\gamma)$ of $G$.
Both papers \cite{bgwexact} and \cite{bew} require $C$ to be unital.
For every action $(B,\alpha)$ of $G$,
first form the diagonal action $\alpha\otimes\gamma$ of $G$
on the maximal tensor product $B\otimes_{\max} C$.
The embedding 
$b\mapsto b\otimes 1$
from $B$ to $B\otimes_{\max} C$
is $G$-equivariant,
and its crossed product is a homomorphism
from $B\rtimes_\alpha G$ to
$(B\otimes_{\max} C)\otimes_{\alpha\otimes\gamma} G$.
The
\emph{$C$-crossed product} $B\rtimes_{\alpha,C} G$
is
the image of $B\rtimes_\alpha G$
in $(B\otimes_{\max} C)\rtimes_{\alpha\otimes\gamma} G$
under this crossed-product homomorphism.
We want an analogue of this construction for coaction functors.
Our previous work indicates that
there should be a coaction on
$B\rtimes_{\alpha,C} G$
that is
the result of applying a coaction functor to
the dual coaction
$(B\rtimes_\alpha G,\what\alpha)$,
and presumably this should involve the fixed dual coaction
$(C\rtimes_\gamma G,\what\gamma)$.
Abstractly, we are led to search for a coaction functor
formed by somehow combining a coaction $(A,\delta)$
with a fixed coaction $(D,\zeta)$,
with $D$ unital,
to form a coaction $(A^D,\delta^D)$
in such a manner that if the two coactions are
$(B\rtimes_\alpha G,\what\alpha)$
and
$(C\rtimes_\gamma G,\what\gamma)$
then $(B\rtimes G)^{C\rtimes G}$ is
the natural image of $B\rtimes_\alpha G$ in
$(B\otimes_{\max} C)\rtimes_{\alpha\otimes\gamma} G$,
and $(\what\alpha)^{C\rtimes G}$ is
the restriction of the dual coaction $\what{\alpha\otimes\gamma}$.
Since we require $C$ to be unital, the crossed product $C\rtimes_\gamma G$ is unital too, so we incur no penalty by supposing that $D$ is unital as well.

We accomplish our goal via a ``$G$-balanced Fell bundle'' $\AA\xg\DD$,
whose cross-sectional $C^*$-algebra
embeds faithfully in the maximal tensor product $A\xm D$.

In \secref{prelim} we record our notation and terminology for
coactions,
Fell bundles,
and coaction functors.

Sections~\ref{tensor D}--\ref{minimal}
contain our main results,
and begin,
as we mentioned above,
by proving in \thmref{sigma D} the existence of
the \emph{tensor $D$ coaction functor},
for a fixed dual coaction $(D,\zeta)$.
For a maximal coaction $(A,\delta)$,
we define an equivariant homomorphism from $A$ to the $G$-balanced tensor product $A\otimes_GD$,
and then for an arbitrary coaction we first compose with maximalization.
In \thmref{compose}
we prove that when we compose with the full crossed product
we recover the tensor-crossed-product functors of \cite{bgwexact,bew}.
We prove in \thmref{min}
that
the case $D=\ell^\infty(G)\rtimes G$,
with $\zeta$ the dual of the translation action,
gives the smallest of these coaction functors.
We point out that our methods are in many cases drawn from those of \cite{bgwexact,bew},
but we modified the proof of minimality
--- \cite{bew} chooses an arbitrary state and temporarily uses completely positive maps as opposed to homomorphisms,
and we managed to avoid the need for these techniques.
Before that, we prove
a general lemma involving embeddings into exact functors,
from which 
in \thmref{exact}
we deduce
that all tensor $D$ coaction functors are exact.

We close in \secref{conclusion} with a few concluding remarks.
First of all, we acknowledge that our standing assumption that the group $G$ is discrete was heavily used,
and we hope to generalize in future work to arbitrary locally compact groups.
We also mention that it is certainly necessary to use a mixture of Fell bundles and coactions --- Fell bundles by themselves are insufficient for our purposes.
We then describe a tantalizing connection with the coaction functors determined by large ideals $E$ of the Fourier-Stieltjes algebra.

We added a very short appendix containing a Fell-bundle version of \lemref{abstract exact},
which could be proved using the lemma and which would lead to a quick proof of \thmref{exact}.
However, we felt that it would interrupt the flow too much to actually use \propref{fb exact} in the main development of \secref{exactness}, and it would in fact have lengthened the exposition.

\section{Preliminaries}\label{prelim}

Throughout, $G$ will be a discrete group,
with identity element denoted by $e$.

We refer to \cite[Appendix~A]{enchilada} and \cite{maximal} for background material on coactions,
and to \cite{klqfunctor,klqfunctor2} for coaction functors.

For an action $(A,\alpha)$ of $G$ we use the following notation:
\begin{itemize}
\item
$(i_A^\alpha,i_G^\alpha)$ is the universal representation of $(A,\alpha)$ in $M(A\rtimes_\alpha G)$;
this is abbreviated to $(i_A,i_G)$ when the action $\alpha$ is clear from context.

\item
$\what\alpha$ is the dual coaction of $G$ on $A\rtimes_\alpha G$.
\end{itemize}

\subsection*{Fell bundles}

We work as much as possible in the context of Fell bundles over $G$,
and the primary references are \cite{fd2, exelbook, discrete}.
The canonical Fell bundle over $G$ is the line bundle $\C\times G$,
whose $C^*$-algebra is naturally isomorphic to $C^*(G)$.
If $\AA=\{A_s\}_{s\in G}$
and $\BB=\{B_s\}_{s\in G}$
are Fell bundles over $G$,
we say a map $\phi\:\AA\to\BB$ is a \emph{homomorphism}
if it preserves all the structure (in particular, multiplication and involution).

We call an operation-preserving map $\pi$ from a Fell bundle $\AA$ into a $C^*$-algebra $B$ 
a \emph{representation} of $\AA$ in $B$.
We call a representation $\phi$ \emph{nondegenerate} if
$\phi(A_e)B=B$.
We write $C^*(\AA)$ for the cross-sectional $C^*$-algebra of a Fell bundle $\AA$,
and $i_\AA\:\AA\to C^*(\AA)$ for the universal representation,
so that for every nondegenerate representation $\pi\:\AA\to B$ there is a unique (nondegenerate) homomorphism
$\wilde\pi\:C^*(\AA)\to B$,
which we call the \emph{integrated form} of $\pi$,
making the diagram
\[
\xymatrix{
\AA \ar[r]^-\pi \ar[d]_{i_\AA}
&B
\\
C^*(\AA) \ar@{-->}[ur]_{\wilde\pi}
}
\]
commute.
If $\phi\:\AA\to \BB$ is a Fell-bundle homomorphism,
then $i_\BB\circ\phi$ is a nondegenerate representation,
so by the universal property
there is a unique $C^*$-homomorphism $C^*(\phi)$ making the diagram
\[
\xymatrix@C+20pt{
\AA \ar[r]^-\phi \ar[d]_{i_\AA}
&\BB \ar[d]^{i_\BB}
\\
C^*(\AA) \ar[r]_-{C^*(\phi)}
&C^*(\BB)
}
\]
commute.
Thus, the assignment $\AA\mapsto C^*(\AA)$ is functorial from Fell bundles to 
the \emph{nondegenerate category} of  $C^*$-algebras,
in which a morphism from $A$ to $B$ is a nondegenerate homomorphism $\pi\:A\to M(B)$.
We frequently
suppress the universal representation $i_\AA$,
and
regard the fibres $A_s$ of the Fell bundle $\AA$ as sitting inside $C^*(\AA)$,
so that the passage from a representation to its integrated form can be regarded as extending from $\AA$ to $C^*(\AA)$,
and in fact we will frequently use the same notation for both the representation and its integrated form.

Recall from \cite{ngdiscrete, discrete}
that 
for every Fell bundle $\AA$ over $G$
there is a coaction $\delta_\AA\:C^*(\AA)\to C^*(\AA)\otimes C^*(G)$
(where unadorned $\otimes$ denotes the minimal $C^*$-tensor product)
given by
\[
\delta_\AA(a_s)=a_s\otimes s\righttext{for}s\in G,a_s\in A_s,
\]
and conversely
for
every coaction $(A,\delta)$ of $G$ the spectral subspaces
\[
A_s=\{a\in A:\delta(a)=a\otimes s\}
\]
give a Fell bundle $\AA=\{A_s\}_{s\in G}$.
Moreover, if $(B,\epsilon)$ is another coaction,
with associated Fell bundle $\BB$,
then
a homomorphism $\phi\:A\to B$ is $\delta-\epsilon$ equivariant
if and only if
it restricts to a homomorphism $\AA\to \BB$.
By \cite[Proposition~4.2]{maximal}, the coaction $\delta$ is maximal if and only if
the integrated form of the inclusion representation $A_s\hookrightarrow A$
is an isomorphism $C^*(\AA)\simeq A$.
Thus, for maximal coactions $(A,\delta)$ we can define a homomorphism from $A$ to another $C^*$-algebra $B$ simply by giving a representation of the Fell bundle $\AA$ in $B$.

\begin{rem}\label{subgroup}
We will need
the following result \cite[Corollary~6.3]{AEKdynamical}:
if $\AA=\{A_s\}_{s\in G}$ is a Fell bundle over $G$
and $H$ is a subgroup of $G$,
then the canonical map $C^*(\AA_H)\to C^*(\AA)$
is injective
(where $\AA_H=\{A_h\}_{h\in H}$ is the restriction to a Fell bundle over $H$).
\end{rem}

\section{Tensor $D$ functors}\label{tensor D}

We will be particularly interested in the case of a homomorphism from
the canonical bundle $\C\times G$ to another Fell bundle
$\DD=\{D_s\}_{s\in G}$,
and we will just say that we have a homomorphism $V\:G\to\DD$.
Note that this will require the unit fibre $C^*$-algebra $D_e$ to be unital, and the elements $V_s$ for $s\in G$ will have to be unitary.

Given Fell bundles $\AA,\DD$ over $G$,
with cross-sectional algebras $A=C^*(\AA)$ and $D=C^*(\DD)$,
we form a new Fell bundle
$\AA\xg\DD$
over $G$ as follows:
the fibre over $s\in G$
is the closure
in $A\xm D$
of the algebraic tensor product $A_s\odot D_s$,
and we write this fibre as $A_s\xm D_s$.
We write
\[
A\xg D=C^*(\AA\xg\DD).
\]

We then define a Fell-bundle homomorphism
\[
\phi_\AA\:\AA\to \AA\xg\DD
\]
by
\[
\phi_\AA(a_s)=a_s\otimes V_s\righttext{for}a_s\in A_s.
\]
Then the image $\phi_\AA(\AA)$ is a 
Fell subbundle
\[
\AA^\DD=\{A_s\otimes V_s\}_{s\in G}
\]
of $\AA\xg\DD$.
Applying the $C^*$-functor gives a homomorphism
\[
Q_A=C^*(\phi_\AA)\:A\to A\xg D,
\]
and we write
\[
A^D=Q_A(A).
\]
Occasionally, if $A$ is understood we will just write $Q$ for $Q_A$.
On the other hand, if $\DD$, and so $D$, is ambiguous, we write $Q_A^D$.
There is a subtlety: although the fibres $A_s\xm D_s$ give a linearly independent family of Banach subspaces of $A\xm D$
with dense linear span,
making $A\xm D$ a graded $C^*$-algebra over $G$ in the sense of Exel \cite[Definition~6.2]{exelbook},
it is not a priori obvious that
the the inclusion map $\AA\xg\DD\hookrightarrow A\xm D$ gives a faithful embedding of
the Fell-bundle $C^*$-algebra $A\xg D$ 
in 
$A\xm D$.
The following theorem establishes this fact.

\begin{thm}\label{faithful}
If $\pi\:\AA\xg\DD\to A\xm D$ is the  representation given by inclusions of the subspaces $A_s\xm D_s$,
the integrated form
\[
\wilde\pi\:A\xg D\to A\xm D
\]
is injective.
\end{thm}

\begin{proof}
First, consider the Fell bundle
\[
\AA\xm\DD=\{A_s\xm D_t\}_{(s,t)\in G\times G},
\]
where, similarly to the definition of $\AA\xg\DD$, we define $A_s\xm D_t$ as the closure of the algebraic tensor product $A_s\odot D_t$ in $A\xm D$.
By \cite[Proposition~4.6]{abadietensorfell} 
the integrated form of
the representation of the Fell bundle $\AA\xm\DD$ in $A\xm D$ given by the
inclusions
\[
A_s\xm D_t\hookrightarrow A\xm D
\]
is an injective homomorphism
\[
C^*(\AA\xm\DD)\to A\xm D.
\]
In view of this, we identify $C^*(\AA\xm\DD)=A\xm D$.

Now, the diagonal subgroup
\[
\Delta=\{(s,s):s\in G\}
\]
of $G\times G$ is isomorphic to $G$ in the obvious way,
and thus the restriction
\[
(\AA\xm\DD)_\Delta=\{A_s\xm D_s\}_{s\in G}
\]
of the Fell bundle $\AA\xm\DD$ to $\Delta$
is canonically isomorphic to our Fell bundle $\AA\xg\DD$ over $G$.
By \cite[Corollary~6.3]{AEKdynamical}
(which we mentioned in \remref{subgroup})
the canonical map
\[
C^*\bigl((\AA\xm\DD)_\Delta\bigr)\to C^*(\AA\xm\DD)
\]
is injective,
and it follows by the above isomorphism that $\wilde\pi$ is injective also.
\end{proof}

In view of \thmref{faithful} we can identify $A\xg D$ with the $C^*$-subalgebra of $A\xm D$ given by the closed span of the subspaces $\{A_s\xm D_s\}_{s\in G}$.

Note that the homomorphism $Q_A\:A\to A\xm D$ is nondegenerate.

For any Fell bundle $\AA$, let $\delta=\delta_\AA$ be the canonical coaction of $G$ on $A=C^*(\AA)$.
By functoriality,
$Q_A$ is
$\delta_\AA-\delta_{\AA\xg\DD}$
equivariant, and hence is
equivariant for $\delta$ and a unique coaction $\delta^D$ on the image $A^D$.

\begin{thm}\label{sigma D}
There is a functor $\sigma^D$ from the category of maximal coactions to the category of all coactions, defined as follows:
\begin{enumerate}
\item
On objects: $(A,\delta)\mapsto (A^D,\delta^D)$.

\item
On morphisms:
given maximal coactions $(A,\delta)$ and $(B,\epsilon)$ and
a $\delta-\epsilon$ equivariant homomorphism $\phi\:A\to B$,
let
$\AA$ and $\BB$ be Fell bundles such that $A=C^*(\AA)$ and $B=C^*(\BB)$,
let
$\psi\:\AA\to\BB$ be the unique Fell-bundle homomorphism such that $\phi=C^*(\psi)$,
and define $\phi^D$ by the commutative diagram
\[
\xymatrix@C+30pt{
A \ar[r]^-\phi \ar[d]_{Q_A}
&B \ar[d]^{Q_B}
\\
A^D \ar@{-->}[r]^-{\phi^D} \ar@{^(->}[d]
&B^D \ar@{^(->}[d]
\\
A\xg D \ar[r]_-{C^*(\psi\xg\id)}
&B\xg D,
}
\]
where it is clear that $C^*(\psi\xg\id)$ maps $A^D$ into $B^D$.

Moreover, for any maximal coaction $(A,\delta)$
we have
\[
\ker Q_A\subset \ker \Lambda,
\]
where $\Lambda\:A\to A^n$
is the normalization.
\end{enumerate}
\end{thm}

\begin{proof}
We obviously have a functor $A\mapsto A\xm D$ on the category of $C^*$-algebras.
As discussed above, if $\delta$ is a maximal coaction on $A$,
then we have a homomorphism $Q_A\:A\to A\xm D$ taking $a_s$ to $a_s\otimes V_s$.
If $\phi\:A\to B$ is equivariant for maximal coactions $\delta$ and $\epsilon$,
then obviously we get a commutative diagram
\[
\xymatrix@C+20pt{
A \ar[r]^-\phi \ar[d]_{Q_A}
&B \ar[d]^{Q_B}
\\
A\xm D \ar[r]_-{\phi\xm\id}
&B\xm D,
}
\]
so
\[
(\phi\xm\id)(A^D)\subset B^D,
\]
and hence the restriction gives a homomorphism $\phi^D$ making the diagram
\[
\xymatrix{
A \ar[r]^-\phi \ar[d]_{Q_A}
&B \ar[d]^{Q_B}
\\
A^D \ar[r]_-{\phi^D}
&B^D
}
\]
commute.
Moreover, by considering elements of spectral subspaces
it is obvious that $\phi^D$ is $\delta^D-\epsilon^D$ equivariant.
Since $A\mapsto A\xm D$ is functorial,
it follows that we now have a functor $A\mapsto A^D$
from maximal coactions to coactions.

We turn to the inclusion $\ker Q_A\subset \ker\Lambda$.
The composition $\delta^D\circ Q_A$
maps $A$ into $(A\xm D)\otimes C^*(G)$.
Composing with the homomorphism
\[
\Upsilon\otimes\lambda\:(A\xm D)\otimes C^*(G)
\to A\otimes D\otimes C^*_r(G),
\]
where $\Upsilon$ is the canonical surjection
\[
A\xm D\to A\otimes D,
\]
we get a homomorphism
\begin{equation}\label{comp}
(\Upsilon\otimes\lambda)\circ\delta^D\circ Q_A\:A\to A\otimes D\otimes C^*_r(G)
\end{equation}
which takes any element $a_s\in A_s$ to
\begin{equation}\label{tensor}
a_s\otimes V_s\otimes \lambda_s.
\end{equation}
Representing faithfully on Hilbert space,
we can apply Fell's absorption trick to the representation $V\otimes\lambda$
to construct an endomorphism $\tau$ of
$A\otimes D\otimes C^*_r(G)$
that takes any element of the form \eqref{tensor}
to
\[
a_s\otimes 1_D\otimes\lambda_s.
\]
Then
\[
\tau\circ(\Upsilon\otimes\lambda)\circ\delta^D\circ Q_A
\:A\to A\otimes 1_D\otimes C^*_r(G)
\]
takes any element $a_s\in A_s$ to
\[
a_s\otimes 1\otimes \lambda_s.
\]
Then composing with the obvious isomorphism
\[
\theta\:A\otimes 1_D\otimes C^*_r(G)\variso A\otimes C^*_r(G),
\]
we get a homomorphism
\[
\theta\circ
\tau\circ(\Upsilon\otimes\lambda)\circ\delta^D\circ Q_A
\:A\to A\otimes C^*_r(G)
\]
taking any element $a_s\in A_s$ to $a_s\otimes\lambda_s$.

On the other hand,
$\Lambda=(\id\otimes\lambda)\circ\delta$,
and for any element $a_s\in A_s$ we have
\[
(\id\otimes\lambda)\circ\delta(a_s)=a_s\otimes\lambda_s.
\]
Thus $\theta\circ\tau\circ(\Upsilon\otimes\lambda)\circ\delta^D\circ Q_A=\Lambda$,
so
\[
\ker Q_A\subset \ker \Lambda,
\]
completing the proof.
\end{proof}

The second part of \thmref{sigma D} justifies the following:
\begin{defn}
We define a coaction functor $\tau^D$ on the category of coactions by
\[
\tau^D=\sigma^D\circ \text{(maximalization)}.
\]
\end{defn}

\begin{prop}\label{bundle iso}
Let $(B,\alpha)$ and $(C,\gamma)$ be two actions of $G$.
Then the map
\[
\phi\:(B\times G)\xg (C\times G)\to (B\xm C)\times G
\]
defined by
\[
\phi\bigl((b,s)\otimes (c,s)\bigr)=(b\otimes c,s)
\]
is a Fell-bundle homomorphism,
and
consequently
\[
C^*(\phi)\:(B\rtimes_\alpha G)\xg (C\rtimes_\gamma G)\to (B\xm C)\rtimes_{\alpha\otimes\gamma} G
\]
is a $C^*$-isomorphism.
\end{prop}

\begin{proof}
The first statement is easily verified,
and for the second we produce an inverse:
define
\[
\pi\:B\xm C\to (B\rtimes_\alpha G)\xg (C\rtimes_\gamma G)
\]
as the unique homomorphism associated to the commuting homomorphisms
$i_B\xm 1$ and $1\xm i_C$,
and define a unitary homomorphism
\[
U\:G\to M\bigl((B\rtimes_\alpha G)\xg (C\rtimes_\gamma G)\bigr)
\]
by
\[
U_s=i_G^\alpha(s)\xm i_G^\gamma(s).
\]
Routine computations show that $(\pi,U)$ is a covariant representation of the action $(B\xm C,\alpha\xm\gamma)$,
so its integrated form gives a homomorphism
\[
\Pi=\pi\times U\:(B\xm C)\rtimes_{\alpha\otimes\gamma} G\to (B\rtimes_\alpha G)\xg (C\rtimes_\gamma G).
\]
One checks without pain,
using the identity
\[
\Pi(b\otimes c,s)=\bigl((b,s)\otimes (c,s)\bigr)\righttext{for all}b\in B,c\in C,s\in G,
\]
that 
$\Pi\circ C^*(\phi)$
is the identity on the Fell bundle $(B\times G)\xg(C\times G)$,
and that 
$C^*(\phi)\circ\Pi$
is the identity on generators $(b\otimes c,s)$ of the Fell bundle $(B\xm C)\times G$,
and hence on the entire Fell bundle.
Thus the associated $C^*$-homomorphisms give inverse isomorphisms,
finishing the proof.
\end{proof}

\begin{rem}
Although we will not need it here, we point out that
the technique used in the above proof can also be used to show
that for actions $(B,G,\alpha)$ and $(C,K,\gamma)$,
\[
(B\times G)\xm (C\times K)\simeq (B\xm C)\times (G\times K),
\]
which in turn implies
\[
(B\rtimes_\alpha G)\xm (C\rtimes_\gamma K)\simeq (B\xm C)\rtimes_{\alpha\otimes\gamma} (G\times K).
\]
\end{rem}

\begin{thm}\label{compose}
Let $(C,\gamma)$ be an action of $G$,
with $C$ unital,
and let $\DD$ be the associated semidirect-product Fell bundle,
with $D=C^*(\DD)=C\rtimes_\gamma G$.
Then
\[
\tau^D\circ \text{\(crossed product\)}
\]
is naturally isomorphic to the $C$-crossed product functor
\[
(B,\alpha)\mapsto B\rtimes_{\alpha,C} G.
\]
\end{thm}

\begin{proof}
Let $(B,\alpha)$ be an action,
and define $\psi\:B\to B\xm C$ by $b\mapsto b\otimes 1$.
In the notation of \propref{bundle iso},
we will show that the diagram
\[
\xymatrix@C+30pt{
B\rtimes_\alpha G \ar[d]_Q \ar[dr]^{\psi\rtimes G}
\\
(B\rtimes_\alpha G)^D \ar@{^(->}[d] \ar@{-->}[r]^-{\simeq}_-{\theta_B}
&B\rtimes_{\alpha,C} G \ar@{^(->}[d]
\\
C^*((B\times G)\xg(C\times G)) \ar[r]^-{\simeq}_-{C^*(\phi)}
&C^*((B\xm C)\times G)
}
\]
commutes at the perimeter and that the bottom isomorphism takes
$(B\rtimes_\alpha G)^D$
onto
$B\rtimes_{\alpha,C} G$,
giving the desired isomorphism $\theta_B$.
For the first, it suffices to compute on the Fell bundle $B\times G$:
\begin{align*}
C^*(\phi)\circ Q(b,s)
&=C^*(\phi)\bigl((b,s)\otimes (1,s)\bigr)
\\&=(b\otimes 1,s)
\\&=(\psi\rtimes G)(b,s).
\end{align*}
This computation also makes it clear that $C^*(\phi)$ maps
$(B\rtimes_\alpha G)^D$
onto
$B\rtimes_{\alpha,C} G$.

We still need to verify naturality:
let $\pi\:(B,\alpha)\to (E,\beta)$
be a morphism of actions.
We must show that the diagram
\[
\xymatrix@C+30pt{
(B\rtimes_\alpha G)^D \ar[r]^-{(\pi\rtimes G)^D} \ar[d]_{\theta_B}
&(E\rtimes_\beta G)^D \ar[d]^{\theta_E}
\\
B\rtimes_{\alpha,C} G \ar[r]_-{\pi\rtimes_C G}
&E\rtimes_{\beta,C} G
}
\]
commutes.
Again, it suffices to compute on the Fell bundle $(B\times G)^D$:
\begin{align*}
(\pi\rtimes_C G)\circ\theta_B\bigl((b,s)\otimes (1,s)\bigr)
&=(\pi\rtimes_C G)(b\otimes 1,s)
\\&=(\pi(b)\otimes 1,s)
\\&=\theta_E\bigl((\pi(b),s)\otimes (1,s)\bigr)
\\&=\theta_E\circ (\pi\rtimes G)^D\bigl((b,s)\otimes (1,s)\bigr).
\qedhere
\end{align*}
\end{proof}

\section{Exactness}\label{exactness}

We now want to show that the tensor $D$ functor is exact.
We separate out the following abstract lemma because we feel that it might be useful in other similar situations.
Actually, we suspect that it is folklore, and we include the proof only for completeness.

\begin{lem}\label{abstract exact}
Let
\[
\xymatrix{
0\ar[r] 
&I\ar[r]^-\phi \ar[d]_\eta
&A\ar[r]^-\psi \ar[d]_\zeta
&B\ar[r] \ar[d]^\omega
&0
\\
0\ar[r] &J\ar[r]_-\pi &C\ar[r]_-\rho &D\ar[r] &0
}
\]
be a commutative diagram of $C^*$-algebras and homomorphisms.
%Suppose that:
%\begin{enumerate}
%\item
%the bottom row is exact;
%\item
%$\phi(I)$ is an ideal of $A$;
%\item
%$\psi$ is surjective;
%\item
%the vertical maps are nondegenerate injections.
%\end{enumerate}
%Then the top row is exact.
%\end{lem}
Suppose that
the bottom row is exact,
$\phi(I)$ is an ideal of $A$,
$\psi$ is surjective,
and the vertical maps are nondegenerate injections.
Then the top row is exact.
\end{lem}

\begin{proof}
The top row is exact at $B$.
Also, $\phi$ is injective because $\pi\circ\eta$ is,
so the top row is exact at $I$;
we must show that it is exact at $A$.
Since
$\omega\circ \psi \circ \phi=\rho\circ \pi\circ \eta=0$
and $\omega$ is injective we have
$\psi\circ\phi=0$.
It remains to show that $\ker \psi\subset \phi(I)$.
Let $a\in \ker \psi$.
Then by commutativity $\zeta(a)\in \ker\rho$,
so by exactness there is $c\in J$ such that
$\zeta(a)=\pi(c)$.
Choose an approximate identity $(e_i)$ for $I$.
Then by nondegeneracy $(\eta(e_i))$ is an approximate identity for $J$.
We have
\begin{align*}
\zeta(a)
&=\pi(c)
\\&=\lim_i \pi\bigl(\eta(e_i)c\bigr)
\\&=\lim_i \pi\circ\eta(e_i)\pi(c)
\\&=\lim_i \zeta\bigl(\phi(e_i)a\bigr)
\\&\in \zeta\bigl(\phi(I)\bigr),
\end{align*}
because $\phi(I)$ is an ideal of $A$.
Since $\zeta$ is an injective homomorphism between $C^*$-algebras,
we get $a\in \phi(I)$,
as desired.
\end{proof}

\begin{thm}\label{exact}
Every tensor $D$
functor $\tau^D$ is exact.
\end{thm}

\begin{proof}
Since maximalization is an exact functor,
it suffices to show that if
\[
\xymatrix{
0\ar[r]
&I \ar[r]^-\phi
&A \ar[r]^-\psi
&B \ar[r]
&0
}
\]
is a short exact sequence of $C^*$-algebras carrying compatible maximal coactions of $G$,
then the image under $\tau^D$ is also exact
(we don't need notation for the coactions, since they will take care of themselves in this proof).
We apply 
\lemref{abstract exact}
to the diagram
\[
\xymatrix@C+10pt{
0\ar[r]
&I^D \ar[r]^{\phi^D} \ar@{^(->}[d]
&A^D \ar[r]^-{\psi^D} \ar@{^(->}[d]
&B^D \ar[r] \ar@{^(->}[d]
&0
\\
0\ar[r]
&I\xm D \ar[r]_{\phi\xm\id}
&A\xm D \ar[r]_-{\psi\xm\id}
&B\xm D \ar[r]
&0.
}
\]
Properties of the maximal tensor product
guarantee that 
the bottom row
is exact,
and we have noted that the vertical inclusion maps are nondegenerate.

We have a commutative diagram
\[
\xymatrix{
I \ar[r]^\phi \ar[d]_{Q_I}
&A \ar[d]^{Q_A}
\\
I^D \ar[r]_-{\phi^D}
&A^D.
}
\]
Since $Q_I$ and $Q_A$ are surjective, $\phi^D(I^D)=Q_A(\phi(I))$ is an ideal of $A^D$.

On the other hand, if $\psi\:A\to B$ is a surjection that is equivariant for maximal coactions,
then the commutative diagram
\[
\xymatrix{
A \ar[r]^-\psi \ar[d]_{Q_A}
&B \ar[d]^{Q_B}
\\
A^D \ar[r]_-{\psi^D}
&B^D
}
\]
shows that $\psi^D$ is surjective since $Q_B\circ \psi$ is.

Thus the hypotheses (i)--(iv) of \lemref{abstract exact} are satisfied, so the conclusion follows.
\end{proof}

\begin{rem}
In \cite[Theorem~4.12]{klqfunctor} we gave necessary and sufficient conditions for a coaction functor to be exact,
expressing it as a quotient of the 
functor maximalization,
which is exact.
However, for the functor $\tau^D$ it turned out to be easier to use 
\lemref{abstract exact},
which was inspired by \cite[proof of Lemma~5.4]{bgwexact}.
\end{rem}

\section{Minimal tensor $D$ functor}\label{minimal}

Recall from 
\cite[Definition~4.7, Lemma~4.8]{klqfunctor} that if $\sigma$ and $\tau$ are coaction functors then
$\tau\le \sigma$ means that
for every coaction $(A,\delta)$ there is a homomorphism
$\Gamma$ making the diagram
\[
\xymatrix{
&A^m \ar[dl]_{q^\sigma_A} \ar[dr]^{q^\tau_A}
\\
A^\sigma \ar@{-->}[rr]_\Gamma
&&A^\tau
}
\]
commute.
If $S$ is a family of coaction functors, and $\tau$ is an element of $S$,
we say that $\tau$ is the \emph{smallest} element of $S$ if $\tau\le\sigma$ for all $\sigma\in S$.

The above partial ordering of coaction functors is compatible with the partial ordering of crossed-product functors
(see \cite[p. 8]{bgwexact}) in the sense that
if $\rho$ and $\mu$ are crossed-product functors associated to coaction functors $\tau$ and $\sigma$, respectively,
then $\tau\le \sigma$ implies $\rho\le \mu$.

\cite[Lemma~9.1]{bew} shows that the smallest of the $C$-crossed-product functors is for
$(C,\gamma)=(\elg,\lt)$ (when $G$ is discrete).
For our purposes it will be more convenient to use right translation, so we replace $\lt$ by $\rt$,
which obviously causes no harm.
The tensor $D$ functor with
$(D,\zeta,V)=(\elg\rtimes_\rt G,\what\rt,i_G)$
reproduces the $\elg$-crossed product upon composing with full crossed product,
so clearly we should expect that the tensor $\elg\rtimes_\rt G$ coaction functor is the smallest among all tensor $D$ functors.
We verify this in \thmref{min} below.

\begin{lem}\label{action tensor}
Let $(C,\gamma)$ be an action of $G$.
Define a homomorphism $\psi\:C\to C\xm \elg=\ell^\infty(G,C)$ by
\[
\psi(c)(s)=\gamma_s(c).
\]
Then $\psi$ is $\gamma-(\id\otimes\rt)$ equivariant,
and the crossed-product homomorphism
\[
\Psi\:C\rtimes_\gamma G
\to (C\xm\elg)\rtimes_{\gamma\otimes\rt} G
=C\xm (\elg\rtimes_{\rt} G)
\]
satisfies
\[
\Psi(i_G^\gamma(s))=1\otimes i_G^{\rt}(s)\righttext{for}s\in G.
\]
\end{lem}

\begin{proof}
Folklore.
\end{proof}

\begin{thm}\label{min}
Let $\RR$ be the semidirect-product Fell bundle of the action
$(\ell^\infty(G),\rt)$,
and let $R=C^*(\RR)=\ell^\infty(G)\rtimes_{\rt} G$,
with unitary homomorphism $W=i_G^\rt\:G\to\RR$.
Then $\tau^R$ is the smallest 
among all tensor $D$ functors.
\end{thm}

\begin{proof}
Let $V\:G\to \DD$ be a homomorphism to a Fell bundle,
and let $D=C^*(\DD)$.
Since by definition
the coaction functors $\tau^D$ and $\tau^R$
are formed by first maximalizing and then applying the surjections $Q^D$ and $Q^R$,
by \cite[Lemma~4.8]{klqfunctor}
it suffices to show that for every maximal coaction $(A,\delta)$
there is a homomorphism $\Gamma$ making the diagram
\begin{equation}\label{gt}
\xymatrix{
&A \ar[dl]_{Q^D} \ar[dr]^{Q^R}
\\
A^D \ar@{-->}[rr]^-{\Gamma}
&&A^R
}
\end{equation}
commute.
By Landstad duality, we can assume that $\DD$ is a semidirect-product 
Fell bundle associated to an action $(C,\gamma)$, and $V=i_G$, so that $D=C\rtimes_\gamma G$.
By \lemref{action tensor}
we have a
homomorphism
\[
\Psi\:D\to C\xm R
\]
such that
\[
\Psi(V_s)=1\otimes W_s.
\]
This gives a homomorphism
\[
\id\xm\Psi\:A\xm D\to A\xm C\xm R
\]
taking $a_s\otimes V_s$ to $a_s\otimes 1\otimes W_s$.
Thus the composition
\[
\Phi=(\id\xm\Psi)\circ Q^D\:A\to A\xm C\xm R
\]
has image in $A\xm 1\xm R$,
and so $\id\xm\Psi$ maps $A^D$ into $A\xm 1\xm R$.
Using the obvious isomorphism
\[
\theta\:A\xm 1\xm R\variso A\xm R,
\]
we see for $s\in G$ and $a_s\in A_s$,
\begin{align*}
\theta\circ(\id\xm\Psi)\circ Q^D(a_s)
&=\theta\circ(\id\xm\Psi)(a_s\otimes V_s)
\\&=\theta(a_s\otimes 1\otimes W_s)
\\&=a_s\otimes W_s
\\&=Q^R(a_s),
\end{align*}
so we can take
\[
\Gamma=\theta\circ(\id\xm\Psi)|_{A^D}.
\qedhere
\]
\end{proof}

The last part of the above proof is very similar to an argument in the proof of \thmref{sigma D}.

\section{Concluding remarks}\label{conclusion}

Throughout, we have taken advantage of our standing assumption that the group $G$ is discrete.
In particular, this allowed us to do almost everything with Fell bundles.
In future research we will investigate the case of arbitrary locally compact $G$.

\begin{rem}\label{fell bundle functor}
It would not be useful to try to do everything in the context of functors from Fell bundles to Fell bundles,
because in our most important construction
\[
\AA\mapsto \AA\xm\DD
\]
the image Fell bundle $\AA^\DD$ is isomorphic to $\AA$.
We must ultimately take the target of the functor to be a $C^*$-algebra to get anything of interest.
\end{rem}

For an equivariant maximal coaction $(D,\zeta,V)$, we will now show how the tensor $D$ coaction functor $\tau^D$ is tantalizingly close to a functor coming from a large ideal $E$ of the Fourier-Stieltjes algebra $B(G)$ (see \cite{klqfunctor}).

Recall that a \emph{large ideal} $E$ is the annihilator of an ideal $I$ of $C^*(G)$ that is 
$\delta_G$-invariant and contained in the kernel of the regular representation,
where invariance means that the quotient map
\[
q_E\:C^*(G)\to C^*_E(G)=C^*(G)/I
\]
takes $\delta_G$ to a coaction on $C^*_E(G)$.
For any maximal coaction $(A,\delta)$ we
let $A^E$ be the quotient of $A$ by the kernel of the composition
$(\id\otimes q_E)\circ\delta$.
Then the quotient map $Q^E=Q^E_A\:A\to A^E$
is equivariant for $\delta$ and a coaction $\delta^E$,
and moreover the assignments $(A,\delta)\mapsto (A^E,\delta^E)$
give a functor from maximal coactions to all coactions.
Composing with the maximalization functor gives a coaction functor that we call \emph{$E$-ization}.

Apply this to the ideal
$I=\ker V$, where
we also write
$V\:C^*(G)\to D$
for
the integrated form of the unitary homomorphism $V\:G\to D$.
The annihilator $E=I\ann$ is a large ideal of $B(G)$ since
$V$ is $\delta_D-\zeta$ invariant and nonzero.
A cursory glance at the situation might lead one to ask,
``Are the tensor $D$ functor and $E$-ization naturally isomorphic?''

One obvious obstruction 
is that
(for maximal coactions)
the tensor $D$ functor goes into a maximal tensor product $A\xm D$, while
$E$-ization 
goes 
into the minimal tensor product $A\otimes C^*_E(G)$.
We can make a closer connection by 
modifying the coaction $\delta$ so that it becomes
a homomorphism $\delta^M$
that makes the diagram
\[
\xymatrix@C+30pt{
A \ar[r]^-{\delta^M} \ar[dr]_\delta
&A\xm C^*(G) \ar[d]^{\psi}
\\
&A\otimes C^*(G)
}
\]
commute, where $\psi$ is the canonical surjection
of the maximal tensor product onto the minimal one,
and 
satisfies the other axioms for a coaction.

Here is a commutative diagram illustrating how the various maps are related:
\[
\xymatrix@C+30pt{
A \ar[r]^-{\delta^M} \ar[d]_{Q^D}
&A\xm C^*(G) \ar[d]^{\id\otimes V}
\\
A\xm D
&A\xm V(C^*(G)), \ar[l]^-{\id\xm\iota}
}
\]
where $\iota\:V(C^*(G))\hookrightarrow D$
is the inclusion map.
Since $V(C^*(G))$ is naturally isomorphic to $C^*_E(G)$,
we see that the tensor $D$ functor seems to be closer to a version
of ``$E$-ization'' but using 
the modified $\delta^M$.
However, there is 
yet another stumbling block:
we do not know whether the homomorphism $\id\xm\iota$ is injective,
due to the mysteries of maximal tensor products.
A bit more succinctly,
we could view the composition
\[
(\id\xm V)\circ \delta^M\:A\to A\xm V(C^*(G))
\]
(preceded by maximalization)
as a 
sort of ``maximalized version'' of $E$-ization,
and then we could ask whether
it is naturally isomorphic to $\tau^D$.

Here is a particularly important special case:

\begin{q}
Is the minimal tensor $D$ functor
(the case $D=R=\ell^\infty(G)\rtimes G$)
isomorphic to a 
maximalized version of $E$-ization as above?
\end{q}

\appendix

\section{Exactness of Fell Bundle Functors}

Although we do not need it, we mention here how the abstract \lemref{abstract exact}
could be used to deduce a corresponding exactness result for Fell-bundle functors,
quite similarly to how we proved \thmref{exact}.

If $\sigma$ is a functor from Fell bundles over $G$
to $C^*$-algebras, in this appendix we will write
$\AA^\sigma$ for the image under $\sigma$ of a Fell bundle $\AA$,
and $\phi^\sigma$ for the image under $\sigma$ of a homomorphism $\phi\:\AA\to\BB$.
Recall from \cite[Definition~21.10]{exelbook} that a Fell bundle $\II$ that is a subbundle of a Fell bundle $\AA$ is called an \emph{ideal} of $\AA$ if
\[
I_sA_t\subset I_{st}
\midtext{and}
A_tI_s\subset I_{ts}
\righttext{for all}s,t\in G.
\]
The following could be proved similarly to \thmref{exact}.

\begin{prop}\label{fb exact}
Let $\sigma$ and $\rho$ be two functors from Fell bundles to $C^*$-algebras,
Assume that:
\begin{itemize}
\item
for every short exact sequence
\[
\xymatrix{
0 \ar[r]
&\II \ar[r]^-\phi
&\AA \ar[r]^-\psi
&\BB \ar[r]
&0
}
\]
of Fell bundles,
$\phi^\sigma(\II^\sigma)$ is an ideal of $\AA^\sigma$
and $\psi^\sigma$ is surjective;

\item
there is a natural transformation $\eta$ from $\sigma$ to $\rho$ such that
for every Fell bundle $\AA$ the homomorphism $\eta_\AA$ maps $A^\sigma_e$ injectively onto a nondegenerate subalgebra of $A^\rho_e$;

\item
$\rho$ is exact.
\end{itemize}
Then $\sigma$ is exact.
\end{prop}

In fact,
\thmref{exact} could be deduced almost immediately from
\propref{fb exact},
but we decided to avoid this approach.

%\bibliographystyle{amsalpha}
%\bibliography{cstar}

\providecommand{\bysame}{\leavevmode\hbox to3em{\hrulefill}\thinspace}
\providecommand{\MR}{\relax\ifhmode\unskip\space\fi MR }
\providecommand{\MRhref}[2]{%
  \href{http://www.ams.org/mathscinet-getitem?mr=#1}{#2}
}
\providecommand{\href}[2]{#2}

\end{document}